\documentclass[a4paper,12pt]{article}
\usepackage{glaudo_pkg_en}

\title{On the \texorpdfstring{$c$-concavity}{c-concavity} with Respect to the Quadratic Cost on a Manifold}
\author{Federico Glaudo\thanks{Scuola Normale Superiore, \url{federico.glaudo@sns.it}.}}

\begin{document}
	\maketitle	
\begin{abstract}
Pushing a little forward an approach proposed by Villani~\cite{Villani08}, we are going to prove that in the Riemannian setting the condition $\nabla^2 f<\metric$ implies that $f$ is $c$-concave with respect to the quadratic cost as soon as it has a sufficiently small $C^1$-norm. From this, we deduce a sufficient condition for the optimality of transport maps.
\blfootnote{\textit{MSC-2010:} 49K, 49Q20, 53C21.}
\end{abstract}

\section{Introduction}
Let us briefly recall the optimal transport problem on $\R^n$ with quadratic cost $c(x,y) = \frac12 d^2(x,y)$.
Given two probability measures $\mu, \nu\in\mathcal P_2(\R^n)$ we want to find a map $T:\R^n\to\R^n$ such that $T_\#\mu = \nu$ and the quantity $\frac12\int \abs{x-T(x)}^2\de\mu(x)$ is minimized.

It is very well-known (see \cite{Brenier91}) that, as soon as $\mu$ and $\nu$ are absolutely continuous with respect to the Lebesgue measure, an optimal map always exists. Moreover a map such that $T_\#\mu = \nu$ is optimal if and only if $T$ is the gradient of a convex function. This striking characterization is a peculiarity of the quadratic cost.

The optimal transport problem has been thoroughly studied in the last three decades (see the monographs \cite{Villani08,AmbrosioGigliSavare08,Santambrogio15}) and for instance a fruitful generalization was the replacement of the space $\R^n$ with a generic Riemannian manifold. In the Riemannian setting the turning point of the theory is given by McCann's Theorem~\cite{McCann01}, that generalizes the theorem of Brenier to a general compact Riemannian manifold. Denoting with $M$ the manifold, McCann proves that there exists an optimal map $T:M\to M$ and furthermore the map $T$ can be written as $\exp(-\nabla f)$, where $f:M\to\R$ is a suitable $c$-concave function.
A $c$-concave function, on a general space $X$ endowed with a symmetric cost function $c:X\times X\to\R$,  is a function $f:X\to\R$ such that there exists a family of $(x_i, \alpha_i)_{i\in I}\subseteq X\times\R$ such that it holds
\begin{equation*}
    f(x) = \inf_{i\in I} c(x, y_i)-\alpha_i \fullstop
\end{equation*}
Of course this is very similar to the definition of concave functions as the infimum of linear functions. Indeed in the Euclidean space, with cost $c=\frac12\abs{x-y}^2$, a function $f$ is $c$-concave if and only if $\frac12x^2-f(x)$ is convex. It is therefore very natural to ask ourselves whether this equivalence can be proven also in the Riemannian setting. The natural generalization, on a compact Riemannian manifold $(M,\metric)$, should look like:
\begin{falsetheorem*}[Na\"{i}ve Statement]
    A function $f:M\to\R$ is $c$-concave if and only if $\nabla^2 f\le \metric$.
\end{falsetheorem*}
Though, this statement does not take into account the fact that the manifold is curved and indeed it turns out being false. We show a counterexample in the last section of this document.

On the other hand, the only result known in literature that goes in this direction is the following, stated by Villani in his monograph~\cite{Villani08} as Theorem 13.5:
\begin{theorem*}[Villani]
    Let $M$ be a compact Riemannian manifold. Then, there is $\eps > 0$ such that any function $\psi\in C^2_c(M)$ satisfying
    $\norm{\psi}_{C^2_b}\le\eps$ is $d^2/2$-concave.
\end{theorem*}
It is immediately clear that such a statement seems not optimal as instead of $\nabla^2 f\le \metric$ it asks that $\nabla^2 f$ is very small. Our main goal is proving a \emph{true} version of the false theorem stated above. The approach is exactly the same as the one adopted by Villani, but instead of using compactness arguments, we deduce explicit inequalities that depend on natural quantities associated to the manifold (curvature, injectivity radius and diameter). 
The exact statement of our main theorem is:
\begin{theorem}[Main Theorem]\label{thm:main_thm}
    Let $(M, \metric)$ be a compact Riemannian manifold with sectional curvature bounded from above by $K\ge 0$. Then there exists a constant $C_*=C_*(\injradius(M), K,\diam(M))>0$ such that, for any $\eps>0$, if $f\in C^2(M, \R)$ is a function with 
    \begin{equation*}
        \norm{\nabla f}_\infty \le \min\left(\frac{\eps}{3K\diam(M)}, C_*\right)
        \hspace{1.5em}\text{and}\hspace{1.5em}
        \nabla^2 f\le (1-\eps)\metric
    \end{equation*}
    then $f$ is $c$-concave.
\end{theorem}
The way this theorem should be intended is that if $\metric-\nabla^2 f\lesssim \norm{\nabla f}_{\infty}\cdot\metric$, then $f$ is $c$-concave.

As a consequence of the main theorem, we give a sufficient condition for a map $T:M\to M$ to be optimal from $\mu$ to $T_\#\mu$. 
The need for such an optimality condition arose naturally while trying to simplify the approach to the random matching problem proposed in \cite{Ambrosio-Stra-Trevisan16}. The simplified approach will appear in the forthcoming paper  \cite{Ambrosio-Glaudo18} coauthored with Ambrosio.

\paragraph{Acknowledgment.} The author warmly thanks professor L.\ Ambrosio for constructive criticism of the manuscript and for several comments and suggestions.

\section{Notations}
Given a compact Riemannian manifold $(M, \metric)$, let us fix the following notation:
\begin{itemize}
    \item Let $d:M\times M\to\co{0}{\infty}$ the Riemannian distance on $M$ and $\exp:TM\to M$ the exponential map.
    \item Let $K$ be the supremum of the positive part of the sectional curvature.
    \item Let $\injradius(M)>0$ be the injectivity radius of the exponential map on $M$.
    \item Let $\diam(M)$ be the diameter of the manifold $M$. We will use that for any Lipschitz function $f:M\to\R$, it holds
    \begin{equation*}
        \sup f - \inf f \le \diam(M) \norm{\nabla f}_\infty \fullstop
    \end{equation*}
    The constant $\diam(M)$ is sharp, indeed if $f(\emptyparam) = d(\emptyparam, \bar x)$ equality can be attained for some $\bar x\in M$.
\end{itemize}

Throughout this notes we will implicitly assume that the cost $c$ is given by $c(x,y)=\frac12d^2(x,y)$. Hence, when we say that a function $f$ is $c$-concave we mean that it can be written as the infimum of functions of the form $x\mapsto\frac12 d^2(x,x_i)-a_i$, where $x_i\in X$, $a_i\in\R$ and $i$ varies in a suitable set of indexes.

\section{Main Theorem}
Exploiting the technical results that we will prove later, the proof of the main theorem becomes straight-forward.
\begin{proof}[Proof of the Main Theorem]
    This is an easy consequence of \cref{prop:tech_version}.
\end{proof}
\begin{remark}\label{rem:convex_easy}
    As said in the introduction, the previous theorem should be seen as a generalization of the trivial fact that if a function $f:\R^n\to\R$ satisfies $\hess f\le \mathds{1}$ then $x\mapsto \frac12 x^2-f$ is convex (indeed on the Euclidean space this convexity is equivalent to the $c$-concavity of $f$).
    
    For a couple of reasons such a statement is harder to prove on a Riemannian manifold. First of all the exponential map need not to be injective globally and that is why we need an additional bound on the gradient. Furthermore, in positive curvature, the Hessian of the square of the distance can be strictly smaller than the metric and consequently we will need to ask a stricter condition on the Hessian of the function itself (i.e. we need $\nabla^2 f$ \emph{strictly} smaller than the metric).
\end{remark}
\begin{remark}
    The theorem can be easily extended to the case where $f$ is compactly supported (and the manifold is non-compact).
\end{remark}
\begin{remark}
    Under the hypothesis that $\norm{\nabla f}_\infty$ is small enough, the requirement $\nabla^2 f<(1+\eps)\metric$ is necessary for $f$ to be $c$-concave. This is a byproduct of the proof of \cref{prop:tech_version}. Indeed, using the notation of that proof, if the global oscillation of $f$ is small enough, the choice of $x^*=\exp_x(-\nabla f(x))$ is mandatory (as we will see in the last section). Therefore it must hold $\nabla^2 h(x)\ge 0$, and that implies the desired bound on the Hessian. 
\end{remark}
\begin{question*}
    Is the assumption $\norm{\nabla f}_\infty=\bigo(\eps)$ optimal?
\end{question*}

The following corollary, deeply linked to McCann's Theorem (see \cite{McCann01}), is the reason behind our investigation of $c$-concave functions.
\begin{corollary}[Optimality Condition]
    Let $M$ be a compact Riemannian manifold. If $f\in C^2(M)$ satisfies the requirements of \cref{thm:main_thm} then, for any probability measure $\mu\in\mathcal{P}(M)$, the map $T=exp(-\nabla f)$ is optimal from $\mu$ to $T_\#\mu$ with respect to the quadratic cost $c=\frac12 d^2$.
\end{corollary}
\begin{proof}
    This is a consequence of the strategy adopted in the proof of \cref{thm:main_thm}. Indeed we will prove that
    \begin{equation*}
        \frac12 d^2(x, \exp_x(-\nabla f(x)))-f(x)\le \frac12 d^2(y, \exp_x(-\nabla f(x)))-f(y)
    \end{equation*}
    for any $x,y\in M$. Given a set of $n$ points $(x_i)_{1\le i \le n}$ and a permutation $\sigma\in S_n$, summing $n$ times the latter inequality we obtain
    \begin{equation*}
        \sum_{i=1}^n \frac12 d^2(x_i, \exp_{x_i}(-\nabla f(x_i))) \le \sum_{i=1}^n \frac12 d^2(x_i, \exp_{x_\sigma(i)}(-\nabla f(x_{\sigma(i)})))
    \end{equation*}
    that proves the $c$-monotonicity of the graph of $T$ and therefore the optimality of $T$.
\end{proof}
\begin{remark}
    The previous corollary could also be deduced directly from the $c$-concavity of the function $f$ using an approach similar to the one used to prove McCann theorem. Anyway we have chosen to give a simpler proof that exploits the equivalence between optimality and $c$-monotonicity.
\end{remark}

\section{Technical Propositions and Proofs}
Let us start stating a well-known characterization of $c$-concavity.
\begin{lemma}\label{lem:concavity_equivalence}
    Let $X$ be a metric space and let $c:X\times X\to\R$ be a symmetric lower-semicontinuous cost.
    A function $f:X\to\R$ is $c$-concave if and only if for any $x\in X$ there exists $x^*\in X$ such that $x\in\argmin\{c(x^*,\emptyparam)-f(\emptyparam)\}$.
\end{lemma}
\begin{proof}
    It is an easy consequence of the fact that $f$ is $c$-concave if and only if for any $x\in X$ there exists $x^*\in X$ such that $f(x)+f^*(x^*) = c(x, x^*)$.
\end{proof}

The two following statements are rather known results in Riemannian geometry. 
The first one is a version of the Hessian comparison that compares a manifold with the constant curvature model, whereas the second is a lower-bound for the convexity radius of a manifold. As a corollary of the Hessian comparison we will obtain a \emph{quantitative} estimate on the Hessian of the square of the distance.

It is not restrictive to assume that $K>0$, and indeed we are going to do it in the following statements, since when $K=0$ all the results can be recovered through a limit procedure.
\begin{theorem}[Hessian Comparison]\label{thm:hess_comp}
    Let us fix a point $x\in M$ and define $r(y)\defeq d(x, y)$.
    At any point $y\in M$ such that $d(x,y)\le\min(\frac{\pi}{\sqrt{K}}, \injradius(M))$ it holds
    \begin{equation*}
        \nabla^2 r \ge \frac{\sqrt{K}\cos(\sqrt{K}r)}{\sin(\sqrt{K}r)}
        \left(\metric - \de r\otimes\de r\right) \fullstop
    \end{equation*}
\end{theorem}
\begin{proof}
    An equivalent, albeit not completely identical, statement can be found at \cite[p.~342]{Petersen98}.
\end{proof}
\begin{corollary}\label{cor:hess_square}
    With the same assumptions of \cref{thm:hess_comp}, if we also have $r<\frac1{\sqrt{K}}$, then it holds
    \begin{equation*}
        \frac12\nabla^2 (r^2) \ge (1-Kr^2)\metric \fullstop
    \end{equation*}
\end{corollary}
\begin{proof}
    Applying the usual calculus rules and \cref{thm:hess_comp} we get
    \begin{equation*}
        \frac12\nabla^2 (r^2) = \de r\otimes\de r + r\nabla^2 r 
        \ge \alpha(\sqrt{K}r)\metric + (1-\alpha(\sqrt{K}r))\de r\otimes\de r \comma
    \end{equation*}
    where $\alpha:\co{0}{\infty}\to\R$ is the function given by $\alpha(t) = \frac{\cos(t)t}{\sin(t)}$.
    The identity $\abs{\nabla r} = 1$ implies $\de r\otimes\de r\le \metric$, therefore we can continue the chain of inequalities and obtain
    \begin{equation*}
        \frac12\nabla^2 (r^2) \ge \left(\alpha(\sqrt{K}r)-\abs{1-\alpha(\sqrt{K}r)}\right)\metric \fullstop
    \end{equation*}
    Hence the thesis follows if we show that for any $0\le t<1$ it holds
    \begin{equation*}
        \abs{1-\alpha(t)}\le \frac{t^2}2 \fullstop
    \end{equation*}
    We leave the proof of this elementary inequality to the reader.
\end{proof}

\begin{theorem}[Convexity Radius Lower-Bound]\label{thm:convexity_radius}
    For any point $x\in M$ and $\delta \le \min\left(\frac{\injradius(M)}2, \frac{\pi}{2\sqrt{K}}\right)$, the ball $B(x, \delta)$ is geodesically convex\footnote{A domain $D\subseteq M$ is geodesically convex if for any $x, y\in D$ there exists a geodesic of length $d(x,y)$ that connects the two points and is completely contained in $D$.}.
\end{theorem}
\begin{proof}
    It can be found at \cite[p.~404]{chavel2006}.
\end{proof}

\begin{proposition}[Technical Version of the Main Theorem]\label{prop:tech_version}
    Given a $C^2$ function $f:M\to\R$, let us denote $\delta = \sqrt{2\diam(M)\norm{\nabla f}_\infty + \norm{\nabla f}^2_\infty}$.
    If $\delta\le \min(\frac{\injradius(M)}2, \frac1{\sqrt{k}})$ and $\nabla^2 f\le (1-K\delta^2)\metric$, then $f$ is $c$-concave.
\end{proposition}
\begin{proof}
    The proof of this proposition is \emph{heavily} inspired by the proof of Theorem 3.15~\cite{Villani08}. Indeed, we are making quantitative the approach proposed by Villani with the help of the Hessian Comparison Theorem.
    
    Let us fix $x\in M$ and define $x^*=\exp_x(-\nabla f(x))$. We are going to prove that $x$ is a minimizer of the function $h(y) \defeq \frac12d^2(x^*, y)-f(y)$. The $c$-concavity follows thanks to \cref{lem:concavity_equivalence}.
    
    More specifically we will show that the three following claims hold:
    \begin{enumerate}
        \item If $y\not\in B(x^*, \delta)$ then $h(y)\ge h(x)$.
        \item It holds $\nabla h(x) = 0$.
        \item In the ball $B(x^*, \delta)$ the function $h$ is convex (i.e. $\nabla^2 h\ge 0$).
    \end{enumerate}
    These three claims imply that $x$ is a global minimizer as we can restrict ourselves in the ball $B(x^*,\delta)$ thanks to 1. and then $x$ is a critical point of a convex function in a convex domain. The convexity of $B(x^*, \delta)$ is a consequence of the assumption on $\delta$ thanks to \cref{thm:convexity_radius}.
    
    Let us prove the three claims separately:
    \begin{description}
     \item[Proof of 1.] If $d(x^*, y)\ge\delta$ we have
        \begin{align*}
            h(y)-h(x) &\ge 
            \frac12 d^2(x^*, y) - (\sup f - \inf f) - \frac12\norm{\nabla f}^2_\infty \\
            & \ge \frac12\delta^2 - \diam(M)\norm{\nabla f}_\infty - \frac12\norm{\nabla f}^2_\infty \ge 0 \fullstop
        \end{align*}
     
     \item[Proof of 2.] The function $d^2(x^*, \emptyparam)$ is smooth in $x$ since $d(x, x^*) = \abs{\nabla f(x)}\le \delta < \injradius(M)$. Hence also the function $h$ is smooth and its gradient is
     \begin{equation*}
        \nabla h(x) = \gamma'(1) - \nabla f(x) \fullstop
     \end{equation*}
    where $\gamma:\cc{0}{1}\to M$ is the constant speed geodesic from $x^*$ to $x$. From the definition of $x^*$ it follows that $\gamma'(1) = \nabla f(x)$ and therefore $\nabla h(x) = 0$.
    
     \item[Proof of 3.] Our assumptions on $\delta$ are exactly those needed to apply \cref{cor:hess_square}, hence, denoting $r(y)=d(x^*, y)$, we get
     \begin{equation*}
        \nabla^2 h = \frac12\nabla^2(r^2)-\nabla^2 f
        \ge (1-Kr^2)\metric-(1-K\delta^2)\metric \ge 0
     \end{equation*}
     that is exactly what we had to show.
    \end{description}
\end{proof}

\section{A Counterexample to the Na\"{i}ve Statement}
In this section we show that the condition $\nabla^2 f\le\metric$ is not sufficient of the $c$-concavity of $f$. We will find a counterexample when the Riemannian manifold is the $2$-dimensional sphere $\S^2$.

Let us start by giving a necessary condition for being $c$-concave.
\begin{proposition}\label{prop:c_conc_necessary}
    There exists a constant $\delta=\delta(M)$ such that, for any $f\in C^1(M)$ with $\osc(f)\le \delta$, the following statements are equivalent:
    \begin{enumerate}
     \item $f$ is $c$-concave;
     \item for any $x\in M$, it holds $x\in\argmin \frac12d^2(x^*, y)-f(y)$ where $x^*=\exp_x(-\nabla f)$.
    \end{enumerate}
\end{proposition}
\begin{proof}
    The implication $2.\implies 1.$ is a straightforward consequence of \cref{lem:concavity_equivalence}.
    
    For the other implication, let us use again \cref{lem:concavity_equivalence} to get that for any $x\in M$ there exists $x^*\in M$ such that $x\in\argmin \frac12d^2(x^*, y)-f(y)$. Thanks to our hypothesis on $\norm{f}_\infty$ we can easily get that any such $x^*$ must be near to $x$ and therefore the distance from $x^*$ has to be smooth at $x$. Hence, given that in $x$ the function $\frac12d^2(x^*, y)-f(y)$ has a minimum, its gradient must be null. 
    Therefore, computing the gradient of the distance function it is easy to prove that it must hold $x^*=\exp_x(-\nabla f)$.
\end{proof}
Our strategy is now to negate that $x\in\argmin \frac12d^2(x^*, y)-f(y)$ looking at the Hessian. Indeed we will build a function such that $\nabla^2 f\le \metric$ but the Hessian of $\frac12d^2(x^*, y)-f(y)$ is not positive-definite at $y=x$.

From now on we will always work on $\S^2$ (of course $\metric$ will denote the Riemannian metric on $\S^2$). The main reason behind this choice is that in this setting the inequality stated in \cref{thm:hess_comp} becomes an identity\footnote{We don't really need to work in dimension $2$ instead of general dimension, but we believe it is much easier to follow a reasoning on $\S^2$ than on a higher dimensional sphere.}.
Let us state explicitly the said identity:
\begin{proposition}[Hessian on the Sphere]\label{prop:hessian_identity}
    Let us fix a point $x\in S^2$ and define $r(y)\defeq d(x, y)$.
    At any point $y\in S^2$ such that $0<d(x,y)<\pi$, it holds
    \begin{equation*}
        \nabla^2 r = \cot(r)
        \left(\metric - \de r\otimes\de r\right)
    \end{equation*}
    and
    \begin{equation*}
        \nabla^2\left(\frac12 r^2\right) = r\cot(r)\cdot\metric + (1-r\cot(r))\de r\otimes\de r \fullstop
    \end{equation*}

\end{proposition}

\begin{proposition}\label{prop:counterexample_property}
    Let $f:\S^2\to\R$ be a $C^2(\S^2)$ function such that $\norm{f}_\infty$ is as small as asked in \cref{prop:c_conc_necessary}. If there exists an $x\in\S^2$ such that $\nabla f(x)\not= 0$ and $\nabla^2 f(x) = \metric$, then $f$ is not $c$-concave.
\end{proposition}
\begin{proof}
    Let us assume that $f$ is $c$-concave. Hence we know from \cref{prop:c_conc_necessary} that the function $y\mapsto\frac12 d^2(x^*, y)-f(y)$ has a global minimum at $y=x$. Hence it must hold that the Hessian of that function is positive semi-definite at $x$. Therefore it must hold
    \begin{equation*}
        \restricts{\nabla^2\left(\frac12 d^2(x^*, y)\right)}{y=x}\ge\metric \fullstop
    \end{equation*}
    However, thanks to \cref{prop:hessian_identity}, it is easy to see that such inequality does not hold if $x^*\not=x$ and that shows the contradiction as $x^*=\exp_x(-\nabla f)$ and $\nabla f(x)\not = 0$.
\end{proof}
It remains to build a function $f:\S^2\to\R$ such that:
\begin{itemize}
 \item $\norm{f}_\infty$ is arbitrarily small.
 \item $\nabla f(N)\not=0$, where $N$ is the north pole of the sphere.
 \item $\nabla^2 f(N)=\metric$. 
 \item Everywhere it holds $\nabla^2 f\le\metric$.
\end{itemize}
Such a function would be a counterexample to the ``na\"{i}ve theorem'' stated in the introduction as it could not be $c$-concave thanks to \cref{prop:counterexample_property}.
We construct our example as a linear combination $f=f_1+\eps f_2$ where $\eps>0$ is a sufficiently small constant and $f_1, f_2$ are such that:
\begin{itemize}
 \item $\norm{f_1}_\infty$ is arbitrarily small.
 \item $\nabla^2 f_1\le\metric$ with equality only in $N$.
 \item $\nabla f_1(N) = 0$.
 \item $\nabla f_2(N) \not=0$, $\nabla^2 f_2(N)= 0$ and $\nabla^2 f_2\le 0$ in a neighbourhood of $N$.
\end{itemize}
It is obvious that if $\eps$ is sufficiently small then $f=f_1+\eps f_2$ satisfies all our requirements.

We are left to prove the existence of $f_1$ and $f_2$ with the said properties.
What we ask on $f_2$ is almost nothing and therefore we leave it to the reader to convince himself that such a function exists.
A good choice for $f_1$ is given by $y\mapsto\rho\left(\frac12d^2(N, y)\right)$ where $\rho:\R\to\R$ satisfies $\rho(t)=t$ in a neighbourhood of $0$ and $\rho$ becomes constant as soon as $t>\eps$ for a certain $\eps>0$. We are not going to perform the computations, but we remark that they are pretty easy exploiting once again that the inequality given by the Hessian comparison is an identity on the sphere.

\bibliographystyle{siam}

\bibliography{biblio}

\end{document}